\documentclass[12pt,twoside]{amsart}
\usepackage{amsmath}
\usepackage{amsthm}
\usepackage{amsfonts}
\usepackage{amssymb}
\usepackage{latexsym}
\usepackage{mathrsfs}
\usepackage{amsmath}
\usepackage{amsthm}
\usepackage{amsfonts}
\usepackage{amssymb}
\usepackage{latexsym}
\usepackage{geometry}
\usepackage{dsfont}
\usepackage[dvips]{graphicx}
\usepackage{color}
\usepackage[all]{xy}

\date{}
\allowdisplaybreaks[4] \footskip=15pt
\renewcommand{\uppercasenonmath}[1]{}

\usepackage{graphicx,amssymb}
\usepackage[all]{xy}
\usepackage{amsmath}

\allowdisplaybreaks
\usepackage{amsthm}
\usepackage{color}

\theoremstyle{plain}
\newtheorem{theorem}{Theorem}[section]
\newtheorem{proposition}[theorem]{Proposition}
\newtheorem{lemma}[theorem]{Lemma}
\newtheorem{corollary}[theorem]{Corollary}
\theoremstyle{definition}
\newtheorem{example}[theorem]{Example}
\newtheorem{definition}[theorem]{Definition}

\theoremstyle{definition}

\theoremstyle{remark}

\newcommand{\pf}{\noindent\begin {proof}}
\newcommand{\epf}{\end{proof}}

\newcommand{\Ker}{\mbox{\rm Ker}}
\newcommand{\Ext}{\mbox{\rm Ext}}
\newcommand{\Hom}{\mbox{\rm Hom}}
\newcommand{\Tor}{\mbox{\rm Tor}}

\newcommand{\Prufer}{Pr\"{u}fer}

\def\ra{\rightarrow}
\def\GV{{\rm GV}}

\def\Hom{{\rm Hom}}
\def\Ext{{\rm Ext}}
\def\Tor{{\rm Tor}}

\def\fkm{{\frak m}}

\def\ker{{\rm ker}}
\def\Ker{{\rm Ker}}

\def\Im{{\rm Im}}
\def\Coker{{\rm Coker}}

\def\Nil{{\rm Nil}}
\def\NN{{\rm NN}}

\def\NP{{\rm NP}}
\def\Z{{\rm Z}}
\def\U{{\rm U}}
\def\T{{\rm T}}
\def\GV{{\rm GV}}
\def\Max{{\rm Max}}
\def\DW{{\rm DW}}
\def\ZN{{\rm ZN}}

\def\PvMR{{\rm PvMR}}
\def\PvMD{{\rm PvMD}}
\def\SM{{\rm SM}}
\def\Krull{{\rm Krull}}

\begin{document}
\begin{center}
{\large  \bf On $\phi$-$w$-Flat modules  and Their Homological Dimensions}

\vspace{0.5cm}   Xiaolei Zhang$^{a}$, Wei Zhao$^{b}$\\

{\footnotesize a.\ Department of Basic Courses, Chengdu Aeronautic Polytechnic, Chengdu 610100, China\\

b.\ School of Mathematics, ABa Teachers University, Wenchuan 623002, China

E-mail: zxlrghj@163.com\\}
\end{center}

\bigskip
\centerline { \bf  Abstract}
\bigskip
\leftskip10truemm \rightskip10truemm \noindent

In this paper,  we introduce and study the class of $\phi$-$w$-flat modules which are generalizations of both $\phi$-flat modules  and $w$-flat modules. The $\phi$-$w$-weak global dimension $\phi$-$w$-w.gl.dim$(R)$ of a strongly $\phi$-ring $R$ is also introduced and studied. We show that,
 for a strongly $\phi$-ring  $R$, $\phi$-$w$-w.gl.dim$(R)=0$ if and only if  $w$-$dim(R)=0$ if and only if $R$ is a $\phi$-von Neumann ring.   It is also proved that, for a strongly $\phi$-ring $R$, $\phi$-$w$-w.gl.dim$(R)\leq 1$ if and only if  each nonnil ideal of $R$ is $\phi$-$w$-flat, if and only if $R$ is a  $\phi$-$\PvMR$,  if and only if $R$ is a $\PvMR$.
\vbox to 0.3cm{}

\leftskip0truemm \rightskip0truemm
\bigskip

Throughout this paper, $R$ denotes a commutative ring with $1\not= 0$ and all modules are unitary. We denote by $\Nil(R)$  the nilpotent radical of $R$, $\Z(R)$ the set of all zero-divisors of $R$ and $\T(R)$ the localization of $R$ at the set of all regular elements. The $R$-submodules $I$ of $\T(R)$ such that $sI\subseteq R$ for some regular element $s$ are said to be \emph{fractional ideals}.
Recall from \cite{A97} that a ring $R$ is  an \emph{$\NP$-ring} if $\Nil(R)$ is a prime ideal, and a \emph{$\ZN$-ring} if $\Z(R)=\Nil(R)$. A prime ideal $P$ is said to be \emph{divided prime} if $P\subsetneq (x)$, for every $x\in R-P$. Set $\mathcal{H}=\{R|R$ is a commutative ring and \Nil(R)\ is a divided prime ideal of $R\}$. A ring $R$ is a \emph{$\phi$-ring} if $R\in \mathcal{H}$. Moreover, a $\ZN$ $\phi$-ring is said to be a \emph{strongly $\phi$-ring}. For a  $\phi$-ring $R$, there is a ring homomorphism $\phi:\T(R)\rightarrow R_{\Nil(R)}$ such that $\phi(a/b)=a/b$ where $a\in R$ and $b$ is a regular element. Denote by the ring $\phi(R)$ the image of $\phi$ restricted to $R$. In 2001, Badawi \cite{A01} investigated  $\phi$-chain rings ($\phi$-CRs for short) and $\phi$-pseudo-valuation rings as a $\phi$-version of chain rings and pseudo-valuation rings. In 2004, Anderson and Badawi \cite{FA04} introduced the concept of $\phi$-\Prufer\ rings and showed that a $\phi$-ring $R$ is $\phi$-\Prufer\ if and only if $R_{\fkm}$ is a $\phi$-chain ring for any maximal ideal $\fkm$ of $R$ if and only if $R/\Nil(R)$ is a \Prufer\ domain if and only if $\phi(R)$ is \Prufer. Later, the authors in \cite{FA05,ALT06} generalized the concepts of Dedekind domains, \Krull\ domains and Mori domains to the context of rings that are in the class $\mathcal{H}$.  In 2013, Zhao et al. \cite{ZWT13} introduced and studied the conceptions of $\phi$-flat modules and $\phi$-von Neumann rings and obtained that a $\phi$-ring is $\phi$-von Neumann if and only if its \Krull\ dimension is $0$.  Recently, Zhao \cite{Z18} gave a homological characterization of $\phi$-\Prufer\ rings as follows: a strongly $\phi$-ring $R$ is $\phi$-\Prufer, if and only if each submodule of a $\phi$-flat module is $\phi$-flat, if and only if each nonnil ideal of $R$ is $\phi$-flat.

Some other important generalizations of classical notions are their $w$-versions. In 1997, Wang and McCasland  {\cite{fm97}} introduced the $w$-modules over strong Mori domains (\SM\ domains for short) which can be seen as a $w$-version of Noetherian domains. In 2011, Yin  \emph{et al}. {\cite{hfxc11}} extended $w$-theories to commutative rings containing zero divisors. The notion of $w$-flat modules appeared first in \cite{f97} for integral domains and was extended to arbitrary commutative rings in \cite{WK15}.  In 2012, Kim and Wang \cite{kf12} introduced $\phi$-\SM\ rings which can be seen as both a $\phi$-version and a $w$-version of Noetherian domains and obtained that a $\phi$-ring $R$ is $\phi$-\SM\ if and only if $R/\Nil(R)$ is an \SM\ domain if and only if $\phi(R)$ is an \SM\ ring. In 2014, Wang and  Kim \cite{KW14} introduced  $w$-w.gl.dim$(R)$ as a generalization of the classical weak global dimension and obtained that a ring $R$ is a von Neumann ring if and only if each $R$-module is $w$-flat, i.e., $w$-w.gl.dim$(R) = 0$. In 2015, Wang and Qiao \cite{fq15} studied several properties of the $w$-weak global dimension, and proved that an integral domain $R$ is a \Prufer\ $v$-multiplication domain ($\PvMD$ for short) if and only if $w$-w.gl.dim$(R) \leq 1$ if and only if $R_{\fkm}$ is a valuation domain for any maximal $w$-ideal $\fkm$ of $R$. As $\phi$-rings are natural extensions of integral domains, we introduce and study the $\phi$-versions of $w$-flat modules, von Neumann rings and $\PvMD$s in this article.  As our work involves $w$-theories, we give a review as below.

Let $R$ be a commutative ring and $J$ a finitely generated ideal of $R$. Then $J$ is called a \emph{$\GV$-ideal} if the natural homomorphism $R\rightarrow \Hom_R(J,R)$ is an isomorphism. The set of all $\GV$-ideals is denoted by $\GV(R)$.
An $R$-module $M$ is said to be \emph{$\GV$-torsion} if for any $x\in M$ there is a $\GV$-ideal $J$ such that $Jx=0$; an $R$-module $M$ is said to be \emph{$\GV$-torsion free} if $Jx=0$, then $x=0$ for any $J\in\GV(R)$ and $x\in M$.  A $\GV$-torsion free module $M$ is said to be a \emph{$w$-module} if  for any $x\in E(M)$ there is a $\GV$-ideal $J$ such that $Jx\subseteq M$ where $E(M)$ is the injective envelope of $M$. The \emph{$w$-envelope} $M_w$ of a $\GV$-torsion free module $M$ is defined  by the minimal $w$-module  that contains $M$. Therefore, a $\GV$-torsion free module $M$ is a $w$-module if and only if $M_w=M$. A \emph{maximal $w$-ideal} for which is maximal among the $w$-submodules of $R$ is proved to be prime (see {\cite[Proposition 3.8]{hfxc11}}). The set of all maximal $w$-ideals is denoted by $ w$-$\Max(R)$. The \emph{$w$-dimension} $w$-$dim(R)$ of a ring $R$ is defined to be the supremum of the heights of all maximal $w$-ideals.

An $R$-homomorphism $f:M\rightarrow N$ is said to be a \emph{$w$-monomorphism} (resp., \emph{$w$-epimorphism}, \emph{$w$-isomorphism}) if for any $ p\in w$-$\Max(R)$, $f_p:M_p\rightarrow N_p$ is a monomorphism (resp., an epimorphism, an isomorphism). Note that $f$ is a $w$-monomorphism (resp., $w$-epimorphism) if and only if $\Ker(f)$ (resp., $\Coker(f)$) is $\GV$-torsion. A sequence $A\rightarrow B\rightarrow C$ is said to be \emph{$w$-exact} if for any $ p\in w$-$\Max(R)$, $A_p\rightarrow B_p\rightarrow C_p$ is exact. A class $\mathcal{C}$ of $R$-modules is said to be \emph{closed under $w$-isomorphisms}  provided that for any $w$-isomorphism $f:M\rightarrow N$, if one of the modules $M$ and $N$ is in $\mathcal{C}$, so is the other. An $R$-module $M$ is said to be of \emph{finite type} if there exist a finitely generated free module $F$ and a $w$-epimorphism $g: F\rightarrow M$, or equivalently, if there exists a
finitely generated $R$-submodule $N$ of $M$ such that $N_w = M_w$. Certainly, the class of  finite type modules is closed under $w$-isomorphisms. Now we proceed to introduce the notion of $\phi$-$w$-flat modules.

\section{$\phi$-$w$-flat modules}

We say an ideal $I$ of $R$ is nonnil provided that there is a non-nilpotent element in  $I$. Denote by $\NN(R)$ the set of all nonnil ideals of $R$. Certainly,  $\GV$-ideals are nonnil. Let $R$  be an $\NP$-ring. It is easy to verify that $\NN(R)$ is a multiplicative system of ideals. That is $R\in \NN(R)$ and for any $I\in \NN(R)$, $J\in \NN(R)$, we have $IJ\in \NN(R)$. Let $M$ be an $R$-module.  Define
\begin{center}
$\phi$-$tor(M)=\{x\in M|Ix=0$ for some  $I\in \NN(R)\}$.
\end{center}
An $R$-module $M$ is said to be \emph{$\phi$-torsion} (resp., \emph{$\phi$-torsion free}) provided that  $\phi$-$tor(M)=M$ (resp., $\phi$-$tor(M)=0$). Clearly, if $R$ is an $\NP$-ring, the class of $\phi$-torsion modules is closed under submodules, quotients, direct sums and direct limits.  Thus an $\NP$-ring $R$ is $\phi$-torsion free if and only if every flat module is $\phi$-torsion free if and only if $R$ is a  $\ZN$-ring (see \cite[Proposition 2.2]{Z18}). The classes of $\phi$-torsion modules and $\phi$-torsion free modules constitute a hereditary torsion theory of finite type. For more details, refer to \cite{S79}.\\
\begin{lemma}\label{w-phi-NN}
Let $R$ be an $\NP$-ring, $\fkm$  a maximal $w$-ideal of $R$ and $I$ an ideal of $R$. Then $I\in \NN(R)$ if and only if $I_{\fkm}\in \NN(R_{\fkm})$.
\end{lemma}
\begin{proof}
Let $I\in \NN(R)$  and $x$  a non-nilpotent element in $I$. We will show the element $x/1$ in $I_{\fkm}$ is a non-nilpotent element of $R_{\fkm}$. If $(x/1)^n=x^n/1=0$ in $R_{\fkm}$ for some positive integer $n$,  there is an $s\in R-\fkm$ such that $sx^n=0$ in $R$. Since $R$ is an $\NP$-ring,  $\Nil(R)$ is the minimal prime $w$-ideal of $R$. In the integral domain $R/\Nil(R)$, we have $\overline{s}\overline{x^n}=\overline{0}$, thus $\overline{x^n}=\overline{0}$ since $s\not\in  \Nil(R)$. So $x\in \Nil(R)$, a contradiction.

Let $x/s$ be a non-nilpotent element in $I_{\fkm}$ where $x\in I$ and $s\in R-{\fkm}$. Clearly, $x$ is non-nilpotent and thus $I\in \NN(R)$.
\end{proof}

\begin{proposition}\label{w-phi-tor}
Let $R$ be an $\NP$-ring, $\fkm$  a maximal $w$-ideal of $R$ and $M$ an $R$-module. Then $M$ is $\phi$-torsion over $R$ if and only $M_{\fkm}$ is $\phi$-torsion over $R_{\fkm}$.
\end{proposition}
\begin{proof} Let $M$ be an $R$-module and $x\in M$. If $M_{\fkm}$ is $\phi$-torsion over $R_{\fkm}$, there is an ideal $I_{\fkm}\in \NN(R_{\fkm})$ such that $I_{\fkm} x/1=0$ in $R_{\fkm}$. Let $I$ be the preimage of $I_{\fkm}$ in $R$. Then $I$ is nonnil  by Lemma \ref{w-phi-NN}. Thus there is a non-nilpotent element $t\in I$  such that $tkx=0$ for some $k\not\in m$. Let $s=tk$. Then we have $(s)\in \NN(R)$ and $(s)x=0$. Thus $M$ is $\phi$-torsion. Suppose $M$ is $\phi$-torsion over $R$.  Let $x/s$ be an element in $M_{\fkm}$. Then there is an ideal $I\in \NN(R)$ such that $Ix=0$, and thus $I_{\fkm}x/s=0$ with $I_{\fkm}\in\Nil(R_{\fkm})$  by Lemma \ref{w-phi-NN}. It follows that $M_{\fkm}$ is $\phi$-torsion over $R_{\fkm}$.
\end{proof}

Recall from \cite{WK15} that an $R$-module $M$ is said to be \emph{$w$-flat} if for any $w$-monomorphism $f: A\rightarrow B$, the induced sequence $f\otimes_R 1:A\otimes_R M\rightarrow B\otimes_R M$ is also a $w$-monomorphism. Obviously, $\GV$-torsion modules and flat modules are all $w$-flat. It was proved that the class of $w$-flat modules is closed under $w$-isomorphisms (see {\cite[Corollary 6.7.4]{fk16}}). Following \cite[Definition 3.1]{ZWT13}, an $R$-module $M$ is said to be \emph{${\phi}$-flat} if for every monomorphism $f : A\rightarrow B$ with $\Coker(f)$ $\phi$-torsion, $f\otimes_R 1 : A\otimes_R M \rightarrow B\otimes_R M$ is a monomorphism. Obviously flat modules are both ${\phi}$-flat and $w$-flat. Now we give a generalization of both ${\phi}$-flat modules and $w$-flat modules.

\begin{definition}\label{w-phi-flat }
Let $R$ be a ring. An $R$-module $M$ is said to be \emph{$\phi$-$w$-flat} if for every monomorphism $f : A\rightarrow B$ with $\Coker(f)$ $\phi$-torsion,  $f\otimes_R 1 : A\otimes_R M \rightarrow B\otimes_R M$ is a $w$-monomorphism; equivalently, if $0 \rightarrow A \rightarrow  B \rightarrow  C \rightarrow 0$ is an exact exact sequence
with $C$ $\phi$-torsion, then $0 \rightarrow A\otimes_R M \rightarrow  B\otimes_R M \rightarrow C\otimes_R M \rightarrow  0$ is
$w$-exact.
\end{definition}
Clearly ${\phi}$-flat modules and $w$-flat modules are $\phi$-$w$-flat.  It is well known that an $R$-module $M$ is flat if and only if the induced homomorphism $1\otimes_R f:M\otimes_R I\rightarrow M\otimes_R R$ is exact for any (finitely generated) ideal $I$,  if and only if the multiplication homomorphism $i:I\otimes_R M\ra IM $ is an isomorphism for any (finitely generated) ideal $I$, if and only if $\Tor^R_1 (R/I, M) = 0$ for any (finitely generated) ideal $I$ of $R$. Some similar characterizations of $w$-flat modules and ${\phi}$-flat modules are given in \cite[Proposition 1.1]{KW14} and \cite[Theorem 3.2]{ZWT13}, respectively. We can also obtain some  similar characterizations of $\phi$-$w$-flat modules.

\begin{theorem}\label{w-phi-flat}
Let $R$ be an $\NP$-ring. The following statements are equivalent for an $R$-module $M$:
\begin{enumerate}
    \item $M$ is $\phi$-$w$-flat;
    \item  $M_{\fkm}$ is  ${\phi}$-flat over $R_{\fkm}$ for all $\fkm\in w$-$Max(R)$;
    \item $\Tor^R_1 (T, M) $ is $\GV$-torsion for all $($finite type$)$ $\phi$-torsion $R$-modules $T$;
    \item $\Tor^R_1 (R/I, M)$ is $\GV$-torsion for all $($finite type$)$ nonnil ideals $I$ of $R$;
     \item   $f\otimes_R 1:I\otimes_R M\rightarrow R\otimes_R M$ is $w$-exact for all $($finite type$)$ nonnil ideals  $I$ of $R$;
     \item    the multiplication homomorphism $i:I\otimes_R M\rightarrow IM $ is a $w$-isomorphism for all $($finite type$)$ ideals $I$;
     \item    let $0\rightarrow K\rightarrow F\rightarrow M\rightarrow 0$ be an exact sequence of $R$-modules, where $F$ is free. Then $(K\cap FI)_w = (IK)_w$ for all $($finite type$)$  nonnil ideals $I$ of $R$.

\end{enumerate}
\end{theorem}

\begin{proof}
$(1)\Rightarrow (2)$: Let $\fkm$ be a maximal $w$-ideal of $R$, $f: A_{\fkm}\rightarrow B_{\fkm}$ an $R_{\fkm}$-homomorphism with $\Coker(f)$  $\phi$-torsion over $R_{\fkm}$. Then $\Coker(f)$ is $\phi$-torsion over $R$  by Proposition \ref{w-phi-tor}. It follows that $f\otimes_{R} M: A_{\fkm}\otimes_{R} M\rightarrow B_{\fkm}\otimes_{R} M$ is a $w$-monomorphism over $R$. Localizing at $\fkm$, we have $f\otimes_{R_{\fkm}} M_{\fkm}: A_{\fkm}\otimes_{R_{\fkm}} M_{\fkm}\rightarrow B_{\fkm}\otimes_{R_{\fkm}} M_{\fkm}$ is a monomorphism over $R_{\fkm}$ since $N_{\fkm}\otimes_{R} M_{\fkm}\cong N_{\fkm}\otimes_{R_{\fkm}} M_{\fkm}$ for any $R$-module $N$. It follows that $M_{\fkm}$ is  ${\phi}$-flat over $R_{\fkm}$.

$(2)\Rightarrow (1)$: Let  $f : A\rightarrow B$ be a monomorphism with $\Coker(f)$ $\phi$-torsion. For any $\fkm\in w$-$Max(R)$, we have $f_{\fkm}: A_{\fkm}\rightarrow B_{\fkm}$ is a monomorphism with  $coker(f_{\fkm})$ $\phi$-torsion over $R_{\fkm}$ by Proposition \ref{w-phi-tor}. Since  $M_{\fkm}$ is  ${\phi}$-flat over $R_{\fkm}$, $f_{\fkm}\otimes_{R_{\fkm}} M_{\fkm}: A_{\fkm}\otimes_{R_{\fkm}} M_{\fkm}\rightarrow B_{\fkm}\otimes_{R_{\fkm}} M_{\fkm}$ is a monomorphism. Thus $f\otimes_{R} M: A\otimes_{R} M\rightarrow B\otimes_{R} M$ is a $w$-monomorphism. Consequently, $M$ is $\phi$-$w$-flat.

 The equivalences of $(2)-(7)$  hold from \cite[Theorem 3.2]{ZWT13} by localizing at all maximal $w$-ideals.
\end{proof}

\begin{corollary}\label{w-closed}
Let $R$ be an $\NP$-ring. The class of $\phi$-$w$-flat modules is closed under $w$-isomorphisms.
\end{corollary}
\begin{proof} Let $f:M\rightarrow N$ be a $w$-isomorphism and $T$  a $\phi$-torsion module.  There exist two exact sequences $0\rightarrow T_1\rightarrow M\rightarrow L\rightarrow 0$ and $0\rightarrow L\rightarrow N\rightarrow T_2\rightarrow 0$ with $T_1$ and $T_2$ $\GV$-torsion.
Considering the induced two long exact sequences $\Tor^R_1(T,T_1)\rightarrow \Tor^R_1 (T,M)\rightarrow \Tor^R_1 (T,L)\rightarrow T\otimes T_1$ and $\Tor^R_2 (T,T_2)\rightarrow \Tor^R_1(T,L)\rightarrow \Tor^R_1(T,N)\rightarrow \Tor^R_1(T,T_2)$, we have $M$ is $\phi$-$w$-flat if and only if $N$ is $\phi$-$w$-flat by Theorem \ref{w-phi-flat}.
\end{proof}

\begin{lemma} \label{nin}
Let $R$ be a $\phi$-ring and  $I$ a nonnil ideal of $R$. Then $\Nil(R)=I\Nil(R)$.
\end{lemma}
\begin{proof} Let $I$ be a nonnil ideal of $R$ with a non-nilpotent element $s\in I$.  Then $\Nil(R)\subseteq (s)$. Thus for any $a\in \Nil(R)$, there exists $b\in R$ such that $a=sb$. Thus $\overline{a}=\overline{s}\overline{b}$ in the integral domain $R/\Nil(R)$. Since $\overline{a}=0$ and $ \overline{s}\not=0$, we have $\overline{b}=0$. So $b\in \Nil(R)$ and then $\Nil(R)\subseteq s\Nil(R)\subseteq I\Nil(R)\subseteq \Nil(R)$. It follows that  $\Nil(R)=I\Nil(R)$.
\end{proof}

\begin{proposition}\label{phi-flat-nil}
Let $R$ be a $\phi$-ring and $M$ an $R$-module. Then $ M/\Nil(R)M$ is $\phi$-flat over $R$ if and only if $M/\Nil(R)M$ is flat over $R/\Nil(R)$. Consequently, $R/\Nil(R)$ is always $\phi$-flat over $R$.
\end{proposition}
\begin{proof}  For the ``only if'' part, let $\overline{I}=I/\Nil(R)$ be an ideal of $\overline{R}=R/\Nil(R)$. If $\overline{I}$ is  zero,  certainly $\Tor_1^{\overline{R}}(\overline{R}/\overline{I}, M/\Nil(R)M)=0$.  Let $\overline{I}$ be a non-zero ideal of $\overline{R}$ with $I\in \NN(R)$. Since $M/\Nil(R)M$ is $\phi$-flat over $R$, $\Tor_1^R(R/I,M/\Nil(R)M)=0$.  By Lemma \ref{nin},  $\Tor_1^R(R/\Nil(R),R/I)\cong I\cap \Nil(R)/I\Nil(R)= \Nil(R)/I\Nil(R)=0$. We have $\Tor_1^{\overline{R}}(\overline{R}/\overline{I}, M/\Nil(R)M)\cong\Tor_1^R(R/I, M/\Nil(R)M)=0$ by change of rings.

For the ``if'' part, let $I$ be a nonnil ideal of $R$. Similarly to the proof of ``only if'' part, since $\Tor_1^R(R/\Nil(R),R/I)=0$, we have $\Tor_1^R(R/I, M/\Nil(R)M)\cong\Tor_1^{\overline{R}}(\overline{R}/\overline{I}, M/\Nil(R)M)=0$. It follows that $M/\Nil(R)M$ is $\phi$-flat over $R$.
\end{proof}

By localizing at all maximal $w$-ideals, we obtain the following corollary.
\begin{corollary}\label{phi-flat-w}
Let $R$ be a $\phi$-ring and $M$ an $R$-module. Then $ M/\Nil(R)M$ is $\phi$-$w$-flat over $R$ if and only if $M/\Nil(R)M$ is $w$-flat over $R/\Nil(R)$.
\end{corollary}
\begin{proof}
See Proposition \ref{phi-flat-nil},   Theorem \ref{w-phi-flat} and \cite[Theorem 3.3]{fk14}.
\end{proof}
Certainly if $R$ is an integral domain, every $\phi$-$w$-flat module is $w$-flat. Conversely, this property characterizes integral domains.
\begin{theorem}\label{phi-flat}
The following statements are equivalent for a $\phi$-ring $R$:
\begin{enumerate}
    \item $R$ is an integral domain;
    \item  every $\phi$-$w$-flat module is $w$-flat;
    \item  every $\phi$-flat module is $w$-flat.
\end{enumerate}
\end{theorem}
\begin{proof} $(1)\Rightarrow (2)\Rightarrow (3)$: Trivial.

$(3)\Rightarrow (1)$:   Let $s$ be a nilpotent element of $R$. Then $$\Tor_1^R(R/(s),R/\Nil(R))\cong(s)\cap\Nil(R)/s\Nil(R)=(s)/s\Nil(R)$$ is $\GV$-torsion since $R/\Nil(R)$ is $w$-flat by (3) and Proposition \ref{phi-flat-nil}. Thus there is a $\GV$-ideal $J$ such that $sJ\subseteq s\Nil(R)$. Since $J$ is a nonnil ideal, $\Nil(R)=J\Nil(R)$ by Lemma \ref{nin}. Thus $sJ\subseteq s\Nil(R)=sJ\Nil(R)\subseteq sJ$. That is, $sJ=sJ\Nil(R)$.  Since $sJ$ is finitely generated, $sJ=0$ by Nakayama's lemma. Since $J\in \GV(R)$, $s\in R$ is $\GV$-torsion free, then $s=0$. Consequently, $\Nil(R)=0$ and $R$ is an integral domain.
\end{proof}

Recall from \cite{fk10} that a ring $R$ is said to be a \emph{$\DW$ ring} if every ideal of $R$ is a $w$-ideal. Then a ring $R$ is a $\DW$
ring if and only if every $R$-module is a $w$-module, if and only if $\GV(R) = \{R\}$ (see \cite[Theorem 3.8]{fk10}).
Certainly if $R$ is a $\DW$ ring, every $\phi$-$w$-flat module is $\phi$-flat. Conversely, this property characterizes $\DW$ rings.
\begin{theorem}\label{w-flat}
 The following statements are equivalent for an $\NP$-ring $R$:
\begin{enumerate}
    \item $R$ is a $\DW$ ring;
    \item  every $\phi$-$w$-flat module is $\phi$-flat;
    \item  every $w$-flat module is $\phi$-flat.
\end{enumerate}
\end{theorem}
\begin{proof} $(1)\Rightarrow (2)\Rightarrow (3)$: Trivial.

$(3)\Rightarrow (1)$: For any $J\in \GV(R)$, $R/J$ is $\GV$-torsion, and thus $w$-flat. By (3), $R/J$ is $\phi$-flat. Since every $\GV$-ideal  $J$ is a nonnil ideal of $R$, we have $\Tor_1^R(R/J,R/J)\cong J/J^2=0$. It follows that $J$ is a finitely generated idempotent ideal,  and thus $J$ is projective. So $J=J_w=R$ by \cite[Exercise 6.10(1)]{fk16} and thus $R$ is a $\DW$ ring by \cite[Theorem 6.3.12]{fk16}.
\end{proof}

Some non-integral domain examples  are provided by the idealization construction $R(+)M$ where $M$ is an $R$-module (see \cite{H88}). We recall this construction. Let  $R(+)M=R\oplus M$ as an $R$-module, and define
\begin{enumerate}
    \item ($r,m$)+($s,n$)=($r+s,m+n$).
    \item  ($r,m$)($s,n$)=($rs,sm+rn$).
\end{enumerate}
Under these definitions, $R(+)M$ becomes a commutative ring with identity. Denote by $(0:_RM)$ the set $\{r\in R| rM=0\}$. Now we compute some examples of $\GV$-ideals of $R(+)M$.
\begin{proposition}\label{ideal}
Let $T$ be a commutative ring and $E$ a  $w$-module over $T$ such that $(0:_TE)=0$. Set $R=T(+)E$. Then $J(+)E$ is a  $\GV$-ideal of $R$ for any $J\in \GV(T)$.
\end{proposition}
\begin{proof}  Let $J$ be a $\GV$-ideal of $T$. Then we claim that $J(+)E\in \GV(R)$. Indeed, since  $T(+)E/J(+)E\cong T/J$, for any $i=0,1$, we have
$$\Ext_{R}^i(T(+)E/J(+)E,R)\cong \Ext_{T}^i(T/J,\Hom_{R}(T,R)).$$   Note that $$\Hom_{R}(T,R)=\Hom_{R}(R/0(+)E,R)\cong 0(+)E \cong  E$$ since $(0:_TE)=0$.
Thus $\Ext_{R}^i(T(+)E/J(+)E,R)\cong \Ext_{T}^i(T/J, E)$ for any $i=0,1$. If $J\in \GV(T)$ then $J(+)E\in \GV(R)$ since $E$ is a $w$-module over $T$.
\end{proof}

Now we give an example to show the notion of $\phi$-$w$-flat modules is a strict generalization of  $\phi$-flat modules and $w$-flat modules.
\begin{example}\label{w-phi-flat-e} Let $D$  be a non-$\DW$ integral domain and $K$ its quotient field. Then $R=D(+)K$ is a $\phi$-ring (see \cite[Remark 1]{FA05}). However, by Proposition \ref{ideal}, $R$ is neither an integral domain nor a $\DW$ ring. Consequently, there is a $\phi$-$w$-flat module over $R$ which is neither  $\phi$-flat nor $w$-flat by Theorem \ref{phi-flat} and Theorem \ref{w-flat}.
\end{example}

\section{homological properties of $\phi$-$w$-flat modules}

Let $R$ be a ring. It is well known that the flat dimension of an $R$-module $M$ is defined as the shortest flat resolution of $M$ and the weak global dimension of $R$ is the supremum of the flat dimensions  of all $R$-modules. The $w$-flat dimension $w$-fd$_R(M)$ of an $R$-module $M$ and $w$-weak global dimension $w$-w.gl.dim$(R)$ of a ring $R$ were introduced and studied in \cite{fq15}.  The following result shows that $\phi$-$w$-flat modules can be characterized by ``higher Tor-funcotrs'' when the base ring is a strongly $\phi$-ring, and so we can investigate homological dimensions in terms of  $\phi$-$w$-flat modules over strongly $\phi$-rings.

\begin{lemma}\label{here}\cite[Theorem 1.6]{QZZ}
Let the  $R$ be  a strongly $\phi$-ring.  Then  an $R$-module $M$ is  $\phi$-$w$-flat if and only if $\Tor^R_{k}(R/I,M)$ is $\GV$-torsion for all nonnil ideal $I$ of $R$ and all $k > 0$
\end{lemma}
\begin{proof} Let $I$ be a nonnil ideal of $R$. Since $R$ is a strongly $\phi$-ring, then $I$ contains a nonzero-divisor $a$. Suppose $M$ is a $\phi$-$w$-flat $R$-module.  Since $a$ is a non-zero-divisor of $R$, $\Tor_n^R(R/\langle a\rangle,M)=0$ for any positive integer $n$.  Then $$\Tor_1^{R/\langle a\rangle}(R/I,M/aM)\cong \Tor_1^{R/\langle a\rangle}(R/I,M\otimes_RR/\langle a\rangle)\cong \Tor_1^{R}(R/I,M)$$ which is $\GV$-torsion. Hence $M/Ma$ is a $w$-flat $R/\langle a\rangle$-module. Consequently, for any $n\geq 1$ we have $$\Tor_n^R(R/I,M)\cong \Tor_n^{R/\langle a\rangle}(R/I,M\otimes_RR/\langle a\rangle)\cong \Tor_n^{R/\langle a\rangle}(R/I,M/aM)$$ is $\GV$-torsion by \cite[Theorem 6.7.2]{fk16}.
\end{proof}

In this section, we always assume $R$ is a strongly $\phi$-ring. We first introduce the notion of $\phi$-$w$-flat dimension of an $R$-module over strongly $\phi$-ring $R$ as follows.

\begin{definition}\label{w-phi-flat }
Let $R$ be a strongly $\phi$-ring and $M$ an $R$-module. We write $\phi$-$w$-fd$_R(M)\leq n$  ($\phi$-$w$-fd abbreviates  \emph{$\phi$-$w$-flat dimension}) if there is a $w$-exact sequence of $R$-modules
$$ 0 \rightarrow F_n \rightarrow ...\rightarrow F_1\rightarrow F_0\rightarrow M\rightarrow 0   \ \ \ \ \ \ \ \ \ \ \ \ \ \ \ \ \ \ \ \ \ \ \ \ \ \ \ \ \ \ \ \ \ \ \ \ \ \ \ (\diamondsuit)$$
where each $F_i$ is  $w$-flat for $i=0,...,n-1$ and $F_n$ is $\phi$-$w$-flat. The $w$-exact sequence $(\diamondsuit)$ is said to be  a $\phi$-$w$-flat $w$-resolution of length $n$ of $M$. If such finite $w$-resolution does not exist, then we say $\phi$-$w$-fd$_R(M)=\infty$; otherwise,  define $\phi$-$w$-fd$_R(M) = n$ if $n$ is the length of the shortest $\phi$-$w$-flat $w$-resolution of $M$.
\end{definition}\label{def-wML}
It is obvious that an $R$-module $M$ is $\phi$-$w$-flat if and only if $\phi$-$w$-fd$_R(M)= 0$. Certainly, $\phi$-$w$-fd$_R(M)\leq w$-fd$_R(M)$. If $R$ is an integral domain, then $\phi$-$w$-fd$_R(M)= w$-fd$_R(M)$

\begin{lemma}\cite[Lemma 2.2]{fq15}\label{big-Tor}
Let $N$ be an $R$-module and $0 \ra A \ra F \ra C \ra 0$ a $w$-exact
sequence of $R$-modules with $F$ a $w$-flat module. Then for any $n > 0$, the induced
map $\Tor^R_{n+1}(C, N) \ra \Tor^R_n(A, N)$ is a $w$-isomorphism. Hence, $\Tor^R_{n+1}(C, N)$
is $\GV$-torsion if and only if so is $\Tor^R_n(A, N)$.
\end{lemma}

\begin{proposition}\label{w-phi-flat d}
Let $R$ be a strongly $\phi$-ring. The following statements are equivalent for an $R$-module $M$:
\begin{enumerate}
    \item $\phi$-$w$-fd$_R(M)\leq n$;
    \item $\Tor^R_{n+k}(M, N)$ is $\GV$-torsion for all $\phi$-torsion $R$-modules $N$ and all $k > 0$;
    \item $\Tor^R_{n+1}(M, N)$ is $\GV$-torsion for all $\phi$-torsion $R$-modules $N$;
        \item $\Tor^R_{n+1}(M, R/I)$ is $\GV$-torsion for all nonnil ideals $I$;
    \item $\Tor^R_{n+1}(M, R/I)$ is $\GV$-torsion for all finite type nonnil ideals $I$;
    \item if $0 \rightarrow F_n \rightarrow ...\rightarrow F_1\rightarrow F_0\rightarrow M\rightarrow 0$ is an exact sequence,
where $F_0, F_1, . . . , F_{n-1}$ are flat $R$-modules, then $F_n$ is $\phi$-$w$-flat;
    \item if $0 \rightarrow F_n \rightarrow ...\rightarrow F_1\rightarrow F_0\rightarrow M\rightarrow 0$ is an $w$-exact sequence,
where $F_0, F_1, . . . , F_{n-1}$ are $w$-flat $R$-modules, then $F_n$ is $\phi$-$w$-flat;
    \item if $0 \rightarrow F_n \rightarrow ...\rightarrow F_1\rightarrow F_0\rightarrow M\rightarrow 0$ is an exact sequence,
where $F_0, F_1, . . . , F_{n-1}$ are $w$-flat $R$-modules, then $F_n$ is $\phi$-$w$-flat;
    \item if $0 \rightarrow F_n \rightarrow ...\rightarrow F_1\rightarrow F_0\rightarrow M\rightarrow 0$ is an $w$-exact sequence,
where $F_0, F_1, . . . , F_{n-1}$ are flat $R$-modules, then $F_n$ is $\phi$-$w$-flat.
\end{enumerate}
\end{proposition}
\begin{proof}
$(1) \Rightarrow(2)$: We prove $(2)$ by induction on $n$. For the case $n = 0$,
 (2) holds by Theorem \ref{w-phi-flat} and Lemma \ref{here} as $M$ is $\phi$-$w$-flat. If $n>0$, then
there is a $w$-exact sequence  $0 \rightarrow F_n \rightarrow ...\rightarrow F_1\rightarrow F_0\rightarrow M\rightarrow 0$,
where each $F_i$ is  $w$-flat for $i=0,...,n-1$ and $F_n$ is $\phi$-$w$-flat. Set $K_0 = \ker(F_0\rightarrow M)$. Then both
$0 \rightarrow  K_0 \rightarrow  F_0 \rightarrow  M \rightarrow  0 $ and $0 \rightarrow  F_n \rightarrow  F_{n-1} \rightarrow...\rightarrow  F_1 \rightarrow  K_0 \rightarrow  0$ are $w$-exact, and $\phi$-$w$-fd$_R(K_0)\leq n-1$. By induction, $\Tor^R_{n-1+k}(K_0, N)$ is $\GV$-torsion
for all $\phi$-torsion $R$-modules $N$ and all $k > 0$. Thus, it follows from Lemma \ref{big-Tor} that $\
\Tor^R_{n+k}(M, N)$ is $\GV$-torsion.

$(2) \Rightarrow(3)\Rightarrow(4)\Rightarrow(5)$:  Trivial.

$(5) \Rightarrow(6)$: Let $K_0 = \ker(F_0 \rightarrow  M)$ and $K_i = \ker(F_i \rightarrow  F_{i-1})$, where
$i = 1, . . . , n-1$. Then $K_{n-1}= F_n$. Since all $F_0, F_1, . . . , F_{n-1}$ are flat,
$\Tor^R_1 (F_n, R/I)\cong \Tor^R_{n+1}(M, R/I)$ is $\GV$-torsion for all finite type nonnil ideal $I$. Hence $F_n$
is a $\phi$-$w$-flat module by Theorem \ref{w-phi-flat}.

$(6) \Rightarrow (1)$: Obvious.

$(3) \Rightarrow(7)$: Set $L_n = F_n$ and $L_i = \Im(F_i \rightarrow  F_{i-1})$, where $i = 1, . . . , n - 1$.
Then both $0 \rightarrow  L_{i+1} \rightarrow  F_i \rightarrow  L_i \rightarrow  0$ and $0 \rightarrow  L_1 \rightarrow  F_0 \rightarrow  M \rightarrow  0$ are $w$-exact sequences. By using Lemma \ref{big-Tor} repeatedly, we can obtain that $\Tor^R_1 (F_n, N)$ is
$\GV$-torsion for all $\phi$-torsion $R$-modules $N$. Thus $F_n$ is $\phi$-$w$-flat.

$(7) \Rightarrow  (8) \Rightarrow  (6), (7) \Rightarrow  (9)$ and $(9) \Rightarrow  (6)$:  Trivial.
\end{proof}

\begin{definition}\label{w-phi-flat }
The \emph{$\phi$-$w$-weak global dimension} of a strongly $\phi$-ring $R$ is defined by
\begin{center}
$\phi$-$w$-w.gl.dim$(R) = \sup\{\phi$-$w$-fd$_R(M) | M $ is an $R$-module$\}.$
\end{center}
\end{definition}\label{def-wML}
Obviously, by definition, $\phi$-$w$-w.gl.dim$(R)\leq w$-w.gl.dim$(R)$.  Notice that if $R$ is an integral domain, then $\phi$-$w$-w.gl.dim$(R)= w$-w.gl.dim$(R)$.

\begin{proposition}\label{w-g-flat}
Let $R$ be a strongly $\phi$-ring. The following statements are equivalent for $R$.
\begin{enumerate}
    \item  $\phi$-$w$-fd$_R(M)\leq n$ for all $R$-modules $M$.
    \item $\Tor^R_{n+k}(M, N)$ is $\GV$-torsion for all $R$-modules $M$ and $\phi$-torsion $N$ and all $k > 0$.
    \item  $\Tor^R_{n+1}(M, N)$ is $\GV$-torsion for all $R$-modules $M$ and $\phi$-torsion $N$.
     \item  $\Tor^R_{n+1}(M, R/I)$ is $\GV$-torsion for all $R$-modules $M$ and nonnil ideals $I$ of $R$.
    \item  $\Tor^R_{n+1}(M, R/I)$ is $\GV$-torsion for all $R$-modules $M$ and finite type nonnil ideals $I$ of $R$.
    \item $\phi$-$w$-fd$_R(R/I)\leq n$ for all  nonnil ideals $I$ of $R$.
    \item $\phi$-$w$-fd$_R(R/I)\leq  n$ for all finite type  nonnil ideals $I$ of $R$.
    \item  $\phi$-$w$-w.gl.dim$(R)\leq  n$.
\end{enumerate}
Consequently, the $\phi$-$w$-weak global dimension of a strongly $\phi$-ring $R$ is determined by the
formulas:
\begin{align*}
\phi\mbox{-}w\mbox{-}w.gl.dim(R)&= \sup \{w\mbox{-}fd_R(R/I) |\ I\ is\ a\  nonnil\ ideal\ of\ R\}\\
&= \sup \{w\mbox{-}fd_R(R/I) |\ I\ is\ a\ finite\ type\ nonnil\ ideal\ of\ R\} .
\end{align*}

\end{proposition}
\begin{proof}

$(1) \Leftrightarrow (8)$ and $(1) \Rightarrow  (6) \Rightarrow  (7) \Rightarrow  (8) $:  Trivial.

$(1) \Rightarrow  (2)$ and $(5) \Rightarrow  (1)$: Follows from Proposition \ref{w-phi-flat d}.

$(2) \Rightarrow  (3)\Rightarrow  (4)\Rightarrow  (5)$:  Trivial.

$(8) \Rightarrow  (1)$: Let $M$ be an $R$-module and $0 \rightarrow F_n \rightarrow ...\rightarrow F_1\rightarrow F_0\rightarrow M\rightarrow 0$ an exact sequence, where $F_0, F_1, . . . , F_{n-1}$ are flat $R$-modules.
To complete the proof, it suffices, by Proposition \ref{w-phi-flat d}, to prove that $F_n$ is
$\phi$-$w$-flat. Let $I$ be a finite type nonnil ideal of $R$.
Thus $\phi$-$w$-fd$_R(R/I)\leq n$ by (8). It follows from Lemma \ref{big-Tor} that $\Tor^R_1 (R/I, F_n)\cong \Tor^R_{n+1}(R/I, M)$
is $\GV$-torsion.
\end{proof}

\section{ Rings with $\phi$-$w$-weak global dimension at most one}

It is well known that a commutative  ring $R$ with weak global dimension $0$ is exactly a \emph{von Neumann regular ring}, equivalently $a\in (a^2)$ for any $a\in R$. It was proved in \cite[Theorem 4.4]{KW14} that a commutative ring $R$ has $w$-weak global dimension $0$, if and only if $a\in (a^2)_w$ for any $a\in R$, if and only if $R_{\fkm}$ is a field for any maximal $w$-ideal $\fkm$ of $R$, if and only if $R$ is a von Neumann regular ring. Recall from  \cite{ZWT13} that a $\phi$-ring $R$ is said to be \emph{$\phi$-von Neumann regular} provided that every $R$-module is $\phi$-flat. A $\phi$-ring $R$  is  $\phi$-von Neumann regular, if and only if there is an element $x\in R$ such that $a = xa^2$ for any non-nilpotent element $a\in R$, if and only if $R/\Nil(R)$ is a von Neumann regular ring, if and only if $R$ is zero-dimensional (see \cite[Theorem 4.1]{ZWT13}). Now, we give some more characterizations of $\phi$-von Neumann regular rings.

\begin{theorem}\label{w-g-flat-1}
Let $R$ be a strongly $\phi$-ring. The following statements are equivalent for $R$:
\begin{enumerate}
    \item $\phi$-$w$-w.gl.dim$(R) =0$;
    \item every $R$-module is $\phi$-$w$-flat;
    \item  $a\in (a^2)_w$ for any non-nilpotent element $a\in R$;
    \item $w$-$dim(R)=0$;
    \item $dim(R)=0$;
    \item  $R$  is  $\phi$-von Neumann regular.
\end{enumerate}
\end{theorem}\

\begin{proof} $(1)\Leftrightarrow (2)$ By definition.

$(2)\Rightarrow (3)$: Let $a$ be a non-nilpotent element in $R$. Then $Ra$ is a nonnil ideal of $R$. It follows that $\Tor_1^R(R/Ra,R/Ra)$ is $\GV$-torsion since $R/Ra$ is $\phi$-torsion and $\phi$-$w$-flat. That is, $Ra/Ra^2$ is $\GV$-torsion, and thus $a\in Ra\subseteq (Ra)_w=(Ra^2)_w$.

$(3)\Rightarrow (4)$: Since $R$ is a $\phi$-ring, $\Nil(R)$ is the minimal prime $w$-ideal $R$. We claim that the ring $\overline{R}_{\overline{\fkm}}:=(R/\Nil(R))_{\fkm/\Nil(R)}$ is a field for any $\fkm\in w$-$Max(R)$. Indeed, let $a$ be a non-nilpotent element in $R$. By (3), $(a)_w=(a^2)_w$. Thus $(a)_{\fkm} =(a^2)_{\fkm}$. We have $(\overline{a})_{\overline{\fkm}}=(\overline{a^2})_{\overline{\fkm}}$ as an ideal  of $\overline{R}_{\overline{\fkm}}$. So $\overline{R}_{\overline{\fkm}}$ is a local von Neumann regular ring, and thus a field. Note that $\overline{R}_{\overline{\fkm}}=R_{\fkm}/\Nil(R_{\fkm})$. It follows that $R_{\fkm}$ is $0$-dimensional (see \cite[Theorem 3.1]{H88}). Thus $w$-$dim(R)=0$.

$(4)\Rightarrow (1)$: By Theorem \ref{w-phi-flat}, we just need to show $\Tor^R_1(R/I,R/J)$ is $\GV$-torsion for all nonnil ideals $I$ and all ideals $J$ of $R$. Since  $R$ is a $\phi$-ring with $w$-$dim(R)=0$, $\Nil (R)$ is the unique maximal $w$-ideal of $R$. We just need to show $\Tor^R_1(R/I,R/J)_{Nil(R)}=0$. That is, $(I\cap J/IJ)_{Nil(R)}=0$.

If $J$ is a nonnil ideal of $R$, there is non-nilpotent elements $s\in I$ and $t\in J$ such that $st\in IJ$. Since $st\not\in \Nil(R)$, $(I\cap J/IJ)_{Nil(R)}=0$. If $J$ is a nilpotent ideal of $R$, $I\cap J=J$. Thus $\Tor^R_1(R/I,R/J)_{\Nil(R)}=(I\cap J/IJ)_{\Nil(R)}=(J/IJ)_{\Nil(R)}$. Let $s$ be a non-nilpotent element in $I$. We have $s(j+IJ)=0+(IJ)$ in $ J/IJ$ for any $j\in J$. Thus $(I\cap J/IJ)_{Nil(R)}=0$.

$(4)\Rightarrow (5)$: By (4), $\Nil(R)$ is the unique $w$-maximal ideal of $R$. If $\Nil(R)$ is a maximal ideal of $R$, $(6)$ holds obviously. Otherwise, there is a non-unit element $a$ which is not nilpotent. Since $(a)$ is not a $\GV$-ideal, there is maximal $w$-ideal $\fkm$ such that $\Nil(R)\subsetneq  (a)\subseteq (a)_w\subseteq \fkm$, Thus $w$-$dim(R)\geq 1$, which is a contradiction.

$(5)\Rightarrow (4)$: Trivial.

$(5)\Leftrightarrow (6)$: See \cite[Theorem 4.1]{ZWT13}.

\end{proof}

Recall from \cite{H88} that a ring $R$ is said to be a \emph{\Prufer\ ring} provided that every finitely generated regular ideal $I$ is invertible, i.e., $II^{-1}=R$ where $I^{-1}=\{x\in T(R)|Ix\subseteq R\}$,  or equivalently, there is a fractional ideal $J$ of $R$ such that $IJ=R$. It is well known that an integral domain is a \Prufer\ domain if and only if the weak global dimension of $R$ $\leq 1$.  Recall that a ring $R$ is said to be a $\PvMR$ if  every finitely generated regular ideal $I$ is $w$-invertible, i.e., $(II^{-1})_w=R$,  or equivalently, there is a fractional ideal $J$ of $R$ such that $(IJ)_w=R$. $\PvMD$s are exactly integral domains which are $\PvMR$s. It is known that an integral domain $R$ is a $\PvMD$ if and only if $R_{\fkm}$ is a valuation domain for each $\fkm\in w$-$Max(R)$ if and only if $w$-w.gl.dim$(R)\leq 1$ (see \cite{KW14,fq15}).

Following \cite{A01}, a $\phi$-ring  $R$ is said to be a \emph{$\phi$-chain ring} ($\phi$-CR for short) if for any $a,b\in R-\Nil(R)$, either $a|b$ or $b|a$ in $R$. A $\phi$-ring  $R$ is said to be a \emph{$\phi$-\Prufer\ ring} if every finitely generated nonnil ideal $I$ is $\phi$-invertible, i.e., $\phi(I)\phi(I^{-1})=\phi(R)$. It follows from \cite[Corollary 2.10]{FA04} that a $\phi$-ring $R$ is $\phi$-\Prufer, if and only if $R_{\fkm}$ is a $\phi$-CR for any maximal ideal $\fkm$ of $R$, if and only if $R/\Nil(R)$ is a \Prufer\ domain, if and only if $\phi(R)$ is \Prufer.  For a strongly $\phi$-ring $R$, Zhao \cite[Theorem 4.3]{Z18} showed that $R$ is a $\phi$-\Prufer\ ring if and only if  all $\phi$-torsion free $R$-modules are $\phi$-flat, if and only if  each submodule of a $\phi$-flat $R$-module is $\phi$-flat, if and only if  each nonnil ideal of $R$ is $\phi$-flat.

Let $R$ be a $\phi$-ring.  Recall from \cite{kf12} that a nonnil ideal $J$ of $R$ is said to be a \emph{$\phi$-$\GV$-ideal} (resp., \emph{$\phi$-$w$-ideal}) of $R$ if $\phi(J)$ is a $\GV$-ideal (resp., $w$-ideal) of $\phi(R)$. A $\phi$-ring $R$ is called a \emph{$\phi$-\SM\ ring} if it satisfies the ACC on $\phi$-$w$-ideals. An ideal $I$ of $R$ is  \emph{$\phi$-$w$-invertible} if $(\phi(I)\phi(I)^{-1})_W=\phi(R)$ where $W$ is the $w$-operation of $\phi(R)$. A $\phi$-ring is  \emph{$\phi$-\Krull}\ provided that any  nonnil ideal is $\phi$-$w$-invertible (see \cite[Theorem 2.23]{kf12}).  By extending $\phi$-\Krull\ rings and $\PvMD$s, we give the definition of $\phi$-\Prufer\ $v$-multiplication rings.

\begin{definition}\label{w-phi-flat }
Let $R$ be a $\phi$-ring. $R$ is said to be a \emph{$\phi$-\Prufer\ $v$-multiplication ring} ($\phi$-$\PvMR$ for short) provided that any finitely generated nonnil ideal is $\phi$-$w$-invertible.
\end{definition}

Now we characterize $\phi$-\Prufer\ multiplication rings in terms of $\phi$-$w$-flat modules.

\begin{theorem}\label{w-g-flat-1}
Let $R$ be a  $\phi$-ring. The following statements are equivalent for $R$:
\begin{enumerate}
   \item $R$ is a $\phi$-$\PvMR$;
   \item $R_{\fkm}$ is a $\phi$-CR for any $\fkm\in w$-$Max(R)$;
    \item $R/\Nil(R)$ is a $\PvMD$;
    \item $\phi(R)$ is a $\PvMR$.
\end{enumerate}
Moreover, if $R$ is a strongly $\phi$-ring, all above are  equivalent to

   \ $(5)$ $R$ is a $\phi$-$w$-w.gl.dim$(R)\leq 1$;

   \ $(6)$  every submodule of a $w$-flat module is $\phi$-$w$-flat;

   \  $(7)$  every submodule of a flat module is $\phi$-$w$-flat;

  \  $(8)$  every  ideal of $R$ is $\phi$-$w$-flat;

  \  $(9)$  every  nonnil ideal of $R$ is $\phi$-$w$-flat;

   \  $(10)$  every finite type nonnil ideal of $R$ is $\phi$-$w$-flat.

\end{theorem}\

\begin{proof} Let $R$ be a  $\phi$-ring. Denote by $W$, $w$ and $\overline{w}$  the $w$-operations of $\phi(R)$, $R$ and $R/\Nil(R)$ respectively. We will prove the equivalences of $(1)-(4)$  and $(5)-(10)$.

$(1)\Rightarrow (4)$: Let $K$ be a finitely generated regular ideal of $\phi(R)$. Then $K=\phi(I)$  for some finitely generated nonnil ideal $I$ of $R$ by \cite[Lemma 2.1]{FA04}. Since $R$ is a $\phi$-$\PvMR$, $(KK^{-1})_W=(\phi(I)\phi(I)^{-1})_W=\phi(R)$. Thus $\phi(R)$ is a $\PvMR$.

$(4)\Rightarrow (1)$:  Let $I$ be a finitely generated nonnil ideal of $R$. We will show $I$ is $\phi$-$w$-invertible. By \cite[Lemma 2.1]{FA04}, $\phi(I)$ is a finitely generated regular ideal of $\phi(R)$. Thus  $(\phi(I)\phi(I)^{-1})_W=\phi(R)$ since $\phi(R)$ is a $\PvMR$.

$(2)\Leftrightarrow (3)$: By \cite[Theorem 3.7, Corollary 2.10]{FA04}, $R_{\fkm}$ is a $\phi$-CR for any $\fkm\in w$-$Max(R)$  if and ony if $R_{\fkm}/Nil(R_{\fkm})=(R/Nil(R))_{\fkm}$ is a valuation domain  for any $\fkm\in w$-$\Max(R)$ if and only if $R/Nil(R)$ is a $\PvMD$ (see \cite[Theorem 4.9]{KW14}).

$(3)\Rightarrow (4)$: Note that $\phi(R)/\Nil(\phi(R))\cong R/\Nil(R)$ is a $\PvMD$ (see \cite[Lemma 2.4]{FA04}). Let $\phi(I)$ be a finitely generated regular ideal of $\phi(R)$. Then, by \cite[Lemma 2.1]{FA04}, $I$ is a nonnil ideal of $R$. Then $\overline{I}=I/Nil(R)$ is $w$-invertible over $\overline{R}=R/\Nil(R)$ by (3). That is, $(\overline{I}\overline{I^{-1}})_{\overline{w}}=\overline{R}$.   There is a $\GV$ ideal $\overline{J}$ of $\overline{R}$ such  $\overline{J}\subseteq \overline{I}\overline{I^{-1}}$ (see \cite[Exercise 6.10(2)]{fk16}). So $J\subseteq  II^{-1}$ where $J$ is a $\phi$-$\GV$ ideal of $R$ by \cite[Lemma 2.3]{kf12}. Thus $\phi(J)\subseteq  \phi(I)\phi(I)^{-1}$. Since $\phi(J)\in \GV(\phi(R))$, $(\phi(I)\phi(I)^{-1})_W=\phi(R)$.

$(4)\Rightarrow (3)$:  Suppose $\phi(R)$ is a $\PvMR$. Let $\overline{I}$ is a finitely generated nonzero ideal of $\overline{R}$. Then $I$ is a nonnil ideal of $R$. Thus  $\phi(I)$ is a finitely generated regular ideal of $\phi(R)$ by \cite[Lemma 2.1]{FA04}. So $(\phi(I)\phi(I)^{-1})_W=\phi(R)$ by (4). Hence $J\subseteq  II^{-1}$ in $R$ for some $\phi$-$\GV$ ideal $J$ of $R$ and thus $\overline{J}\subseteq \overline{I}\overline{I^{-1}}$ in $\overline{R}$. By \cite[Lemma 2.3]{kf12}, $\overline{J}\in \GV(\overline{R})$, and thus $(\overline{I}\overline{I^{-1}})_{\overline{w}}=\overline{R}$. So $R/\Nil(R)$ is a $\PvMD$.

 $(5)\Rightarrow (6)$: Let $K$ be a submodule of a $w$-flat module $F$. Then $\phi$-$w$-fd$_R(F/K)\leq 1$ by (5). Thus $K$ is $\phi$-$w$-flat by Proposition \ref{w-phi-flat d}.

$(6)\Rightarrow (7)\Rightarrow (8)\Rightarrow (9)\Rightarrow (10)$: Trivial.

$(10)\Rightarrow (5)$: Let $I$ be a finite type nonnil ideal of $R$. Then $\phi$-$w$-fd$_R(R/I)\leq 1$ by Proposition \ref{w-phi-flat d}. It follows from Proposition \ref{w-g-flat} that  $\phi$-$w$-w.gl.dim$(R)\leq 1$.

Now, let $R$ be a strongly $\phi$-ring.

$(2) \Rightarrow (9)$: Let $\fkm$ be a maximal $w$-ideal of $R$ and $I$ a nonnil ideal of $R$. Then $I_{\fkm}$ is a nonnil ideal of $R_{\fkm}$ by Lemma \ref{w-phi-NN} and thus is $\phi$-flat by  \cite[Theorem 4.3]{Z18}. So $I$ is $\phi$-$w$-flat by Theorem \ref{w-phi-flat}.

$(9) \Rightarrow (2)$: Let  $\fkm$ be a maximal $w$-ideal of $R$ , $I_{\fkm}$  a nonnil ideal of $R_{\fkm}$. Then $I$ is a nonnil ideal of $R$ by Lemma \ref{w-phi-NN}. By (9), $I$ is $\phi$-$w$-flat and so $I_{\fkm}$ is $\phi$-flat by Theorem \ref{w-phi-flat}. Thus $R_{\fkm}$ is a $\phi$-CR by \cite[Theorem 4.3]{Z18}.
\end{proof}

\begin{corollary}\label{w-g-flat-pvmr}
 Suppose $R$ is a  $\phi$-ring. Then $R$ is a $\phi$-\Krull\ ring if and only if $R$ is both a $\phi$-$\PvMR$ and a $\phi$-\SM\ ring.
\end{corollary}\
\begin{proof}
By  \cite[Theorem 2.4]{kf12} a  $\phi$-ring $R$ is a $\phi$-\SM\ ring if and only if $R/\Nil(R)$ is an  \SM\ domain. A  $\phi$-ring $R$ is a $\phi$-\Krull\ ring if and only if $R/\Nil(R)$ is a \Krull\ domain (see \cite[Theorem 3.1]{FA05}). Since  $R$ is a \Krull\ domain if and only if $R$ is an \SM\ $\PvMD$ (see \cite[Theorem 7.9.3]{fk14}), the equivalence holds by Theorem \ref{w-g-flat-1}.
\end{proof}

\begin{corollary}\label{w-g-flat-pvmr}
 Suppose $R$ is a strongly $\phi$-ring. Then $R$ is a $\phi$-$\PvMR$ if and only if $R$ is a $\PvMR$.
\end{corollary}\

\begin{proof}  Suppose $R$ is a $\phi$-$\PvMR$ and let $I$ be a finitely generated regular ideal of $R$. Then $\overline{I}$ is a finitely generated regular ideal of $\overline{R}$. By Theorem \ref{w-g-flat-1}, $\overline{R}$ is a $\PvMD$. Then  $(\overline{I}\overline{I^{-1}})_{\overline{w}}=\overline{R}$. Thus there is a $\GV$-ideal $\overline{J}$ of $\overline{R}$ with $J$ finitely generated over $R$ such that $\overline{J}\subseteq  \overline{I}\overline{I^{-1}}$. Since $R$ is a strongly $\phi$-ring, $J$ is a $\GV$-ideal of $R$ by \cite[Lemma 2.11]{kf12}. Since $J\subseteq  II^{-1}$ in $R$, $(II^{-1})_w=R$. Assume $R$ is a $\PvMR$. Since $R$ is a strongly $\phi$-ring, $\phi(R)=R$ is a $\PvMR$. Thus $R$ is a $\phi$-$\PvMR$ by Theorem \ref{w-g-flat-1}.
\end{proof}
The condition that $R$ is a strongly $\phi$-ring in  Corollary \ref{w-g-flat-pvmr} can't be removed  by the following example.
\begin{example}\label{w-phi-prufer} Let $D$ be an integral domain which is not a $\PvMD$ and $K$ its quotient field. Since $K/D$ is a divisible $D$-module, the ring $R=D(+)K/D$ is a $\phi$-ring but not a strongly $\phi$-ring (see \cite[Remark 1]{FA05}).  Since $\Nil(R)=0(+)K/D$, we have $R/\Nil(R)\cong D$ is not a $\PvMD$. Thus $R$ is not a $\phi$-$\PvMR$ by Theorem \ref{w-g-flat-1}. Denote by $U(R)$ and $U(D)$ the sets of unit elements of $R$ and $D$ respectively. Since $\Z(R)=\{(r,m)|r\in \Z(D)\cup \Z(K/D)\}=R-\U(D)(+)K/D=R-\U(R)$,  $R$ is a $\PvMR$ obviously.
\end{example}

\end{document}